\documentclass[12pt,sumlimits]{amsart}
\usepackage{hyperref}
\usepackage{amssymb}
\usepackage[usenames,dvipsnames,svgnames,table]{xcolor}
\usepackage{amsmath,amssymb,amscd}

\usepackage[12hr]{datetime}

\def\R{\mathbb{R}}

\newcommand{\der}[3][]{\frac{\partial^{#1} #2}{\partial #3^{#1}}}
\newcommand{\diff}[3][]{\frac{d^{#1} #2}{d #3^{#1}}}
\def\bb{\begin{equation}}
\def\ee{\end{equation}}
\def\bse{\begin{subequations}}
\def\ese{\end{subequations}}

\def\pa{\partial}

\newcommand{\la}{\langle}
\newcommand{\ra}{\rangle}
\def\rar{\rightarrow}
\def\ep{\varepsilon}
\def\ls{\lesssim}

\newcommand{\G}[3]{\Gamma_{#1}{}^{#2}{}_{#3}}

\def\bl{{\underline{L}}}
\def\lb{{\underline{L}}}
\def\bL{\bl}

\def\Bl{\bl}

\def\lm{\lambda}
\def\al{\alpha}
\def\be{\beta}
\def\ga{\gamma}
\def\de{\delta}

\def\Ga{\Gamma}

\newtheorem{defin}{Definition}[section]
\newtheorem{lemma}[defin]{Lemma}
\newtheorem{prop}[defin]{Proposition}
\newtheorem{theorem}[defin]{Theorem}

\newtheorem{claim}[defin]{Claim}
\newtheorem*{rema}{Remark}

\numberwithin{equation}{subsection}

\def\paxmt{ (\partial_{x_1}-\partial_{t}) }
\def\patpx{ (\partial_t+\partial_{x_1}) }

\def\fc{\footnotemark[1]}
\def\fl{\footnotemark[2]}
\def\fa{\footnotemark[3]}

\def\rl{R^{\text{lin}}}
\def\rq{R^{\text{quad}}}
\def\rlo{R^{\text{lin}(1)}}
\def\Gt{\tilde{\Gamma}}
\def\rco{\text{Ric}^{(1)}}
\def\Box{\square}\def\Boxr{\widetilde{\square}}\def\pa{\partial}
\def\a{\alpha}\def\b{\beta}

\def\sls#1{\text{$#1\mkern -13.0mu$\slash\,}}

\def\gv{\sqrt{|g|}}

\title[Counterexample to local existence for Einstein's eq.   ]{ A sharp counterexample to local existence of low regularity solutions to Einstein's equations in wave coordinates.}
\author{Boris Ettinger}
\email{boris.ettinger@gmail.com}
\author{Hans Lindblad}
\address{Department of Mathematics, Johns Hopkins University, 404 Krieger Hall, 3400 N. Charles
Street, Baltimore, Maryland 21218}
\email{lindblad@math.jhu.edu}
\begin{document}
\maketitle

%\today, \currenttime
\begin{abstract} We give a sharp counter example to local existence of low regularity solutions to Einstein's equations in wave coordinates. We show that there are initial data in $H^2$ satisfying the wave coordinate condition such that there is no solution in $H^2$ to Einstein's equations in wave coordinates for any positive time.
This result is sharp since Klainerman-Rodnianski and Smith-Tataru proved existence for the same equations with slightly more regular initial data.
\end{abstract}

\section{Introduction}
The Einstein vacuum equations $R_{\mu\nu}=0$ in wave coordinates
becomes a system on nonlinear wave equations, called the
reduced Einstein equations
\begin{equation}\label{eq:reducedEinstein}
\widetilde{\Box}_g g_{\mu\nu} =F_{\mu\nu}(g)[\partial g,\partial g].
\end{equation}
The metric in addition is assumed to satisfy the wave coordinate condition
\begin{equation}\label{eq:WaveCordinateCond}
\partial_\alpha \big(\sqrt{|g|} g^{\alpha\beta}\big)=0,\qquad \text{where} \quad  |g|=|\det{\big(\partial g/\partial x\big)}|,
\end{equation}
which is preserved by the reduced equations if its satisfied initially and if data satisfies the so called constraint equations.
Here $F_{\mu\nu}(g)[\partial g,\partial g]$ are quadratic forms in $\partial g$ with coefficients depending on $g$ and the reduced wave operator is given by
\begin{equation}
\widetilde{\Box}_g=g^{\alpha\beta}\partial_\alpha\partial_\beta.
\end{equation}

We are considering the initial value problem with low regularity data. Given initial data in Sobolev spaces $H^s$;
\begin{equation}
g\big|_{t=0}=g^0\in H^s,\qquad \partial_t g\big|_{t=0} =g^1\in H^{s-1}
\end{equation}
we are asking for which $s$ we can obtain a local solution in $H^s$, i.e.
\begin{equation}
g(t,\cdot)\in H^s,\qquad \partial_t g(t,\cdot)\in H^{s-1}, \quad 0\leq t\leq T,
\end{equation}
for some $T>0$,
given that initial data satisfy the constraint equations and the wave coordinate condition. In 1952 Choquet-Bruhat proved that this is true for large $s$. More recently
Klainerman-Rodinianski \cite{KR} respectively Smith-Tataru \cite{ST} proved local existence in $H^s$, for $s>2$ for Einstein's equations in wave coordinates.
The result in \cite{ST} is in fact for more general quasilinear equations of the above form (see also a recent work of Wang \cite{W}). Moreover, Klainerman-Rodnianski-Szeftel \cite{KRS} recently proved that one has local existence of bounded curvature solutions to Einstein's equations if the curvature is bounded initially. However, that does not imply existence in wave coordinates.

We in fact show that one do not in general have local existence
in $H^2$ for Einstein's equations in wave coordinates:
\begin{theorem} For any $\varepsilon>0$ there is domain of dependence $D$ and a smooth solution to Einstein's equations in wave coordinates in $D$ such that
\begin{equation}
\| g(0,\cdot)-m\|_{H^2(D_0)}+\|\partial_t g(0,\cdot)\|_{H^1(D_0)}\leq \varepsilon,
\end{equation}
where $m$ is the Minkowski metric,
but for any $t>0$
\begin{equation}
\| g(t,\cdot)\|_{H^2(D_t)}+\|\partial_t g(t,\cdot)\|_{H^1(D_t)}=\infty,
\end{equation}
where $D_t=\{x;\, (t,x)\in D\}$. Moreover the curvature tensor satisfies
\begin{equation}
\| R(t,\cdot)\|_{L^2(D_t)}\leq C\varepsilon,
\end{equation}
for any $t$. (Here domain of dependence is an open subset of the upper half space such that the backward light cone from any point in it is also contained in it.)
 \end{theorem}
\begin{rema} By a recent result Czimek \cite{C} data as above can be extended to data on
$\bold{R}^3$ in $H^2$ satisfying the constraint equations and the wave coordinate condition.
\end{rema}

To put the result in the theorem in context we recall that in Lindblad \cite{L1,L2} counterexamples to local existence in $H^2$ where given for the semi-linear equation
\begin{equation}\label{eq:semilinearcounterexample}
\Box \phi=(\bl \phi)^2
\end{equation}
respectively for the quasi-linear equation
\begin{equation}\label{eq:quasilinearcounterexample}
\Box \phi=\phi\,  \bl^2 \phi
\end{equation}
where $\bl=\pa_t-\pa_{x_1}$. The counterexample for the semi-linear equation is much stronger and the quasi-linear counterexample is just due to concentration of characteristics. On the other hand it was shown in Klainerman-Machedon \cite{KM} that there is local existence $H^s$, for any $s>3/2$, for systems that satisfy the null condition, in particular for
\begin{equation}
\Box \phi =(\pa_t\phi)^2-|\nabla_x\phi|^2.
\end{equation}
Einstein's equations in wave coordinates do not satisfy the null condition.
However as was shown in Lindblad-Rodnianski \cite{LR} it satisfy a weak null condition in a null frame and the semi-linear terms can be modelled by the system
\begin{equation}\label{eq:modelcounterexample}
\Box \phi_2=-(\bl \phi_1)^2,\qquad \Box \phi_1=0
\end{equation}
that satisfy the weak null condition.
The same argument used to give a counterexample for the systems
\eqref{eq:semilinearcounterexample} and \eqref{eq:quasilinearcounterexample}
in $H^2$ also gives a counterexample in $H^2$ for the model problem
\eqref{eq:modelcounterexample}:
\begin{prop} For any $\varepsilon>0$ there is a smooth solution $\phi=(\phi_1,\phi_2)$ to \eqref{eq:modelcounterexample} in $D=\{(t,x);\, (x_1-1)^2+x_2^2+x_3^2<(1-t)^2\}$ such that
\begin{equation}\label{eq:smalldata}
\| \phi(0,\cdot)\|_{H^2(D_0)}+\|\partial_t \phi(0,\cdot)\|_{H^1(D_0)}\leq \varepsilon
\end{equation}
but for any $t>0$
\begin{equation}\label{eq:largesolution}
\| \phi(t,\cdot)\|_{H^2(D_t)}+\|\partial_t \phi(t,\cdot)\|_{H^1(D_t)}=\infty,
\end{equation}
where $D_t=\{x;\, (t,x)\in D\}$. Moreover the data can be extended so that
\begin{equation}
\| \phi(0,\cdot)\|_{H^2(\bold{R}^3)}+\|\partial_t \phi(0,\cdot)\|_{H^1(\bold{R}^3)}\lesssim \varepsilon.
\end{equation}
\end{prop}
The proof of this is by finding explicit solutions of the system
depending on $(t,x_1)$ only inside the domain of dependence $D$, that satisfy the conditions. Its easy to check that
for any function $\chi_1$,
$$
\phi_1(t,x)=\chi_1(x_1-t),\qquad \phi_2(t,x)=-t\chi_2(x_1-t)
$$
solves the system if
$$
\chi_2(x_1)=2\int_0^{x_1}\chi_1^\prime(s)^2 \,ds.
$$
Let
$$
\chi_1(x_1)=\int_0^{x_1}{\epsilon|\log{|s/4|}|^\alpha \,ds},\quad
1/4<\alpha<1/2,
$$
in which case
$$
\chi_2(x_1) =2\int_0^{x_1}{\epsilon^2|\log{|s/4|}|^{2\alpha} \,ds}.
$$
We have
\begin{equation}
\|\phi_1(t,\cdot)\|_{H^2(D_t)}\sim \|\chi_1^{\prime\prime}\|_{L^2(D_t)},\qquad
\|\phi_2(t,\cdot)\|_{H^2(D_t)}\sim t\|\chi_2^{\prime\prime}\|_{L^2(D_t)},
\end{equation}
and a  calculation shows that
\begin{equation}
\int_{D_t} \chi_i^{\prime\prime}(t-x_1)^2\, dx
\sim \int_t^{2-t}{|\chi_i^{\prime\prime}(x_1-t)|^2 (x_1-t)\,dx_1}
\,\,\,\begin{cases} <\infty,\quad &\text{if } i=1,\\ =\infty,\quad &\text{if }i=2\end{cases},
\end{equation}
from which the first part of the proposition follows. The second part of the proposition is obtained by multiplying with a cutoff $\chi\big( (x_2^2+x_3^2)/x_1\big)$.

Note that in the example derivatives tangential to the characteristic surfaces $t-x_1=c$ are better behaved than transversal derivatives.

Modulo terms that satisfy the null condition or cubic terms that are smaller because of the smallness in the construction above we have
\begin{equation}\label{eq:ReducedEinstein}
\Boxr_g g_{\mu\nu} \sim P(\pa_\mu g, \pa_\nu g),\quad \text{where}\quad
 P(h,k)
=\frac 14 h^{\a}_\a k^\b_\b-
\frac 12 h^{\a\b} k_{\a\b}.
\end{equation}
Expressing this in a nullframe $L=\pa_t+\pa_{x_1}$, $\lb=\pa_t-\pa_{x_1}$, $A,B=\pa_{x_2},\pa_{x_3}$:
\begin{align}
\Boxr_g g_{TU}&\sim 0,\quad T\in \{L,A,B\}, \, \, U\in \{\lb,L,A,B\}, \label{eq:redwave1} \\
 \Boxr_g g_{\lb\lb}&\sim P(\pa_\lb g,\pa_\lb g)\label{eq:redwave2}.
\end{align}
The linearized version of the wave coordinate condition reads
\begin{equation}
-m^{\mu\nu}\pa_\mu g_{\nu\ga}+\frac{1}{2}m^{\mu\nu}\,\pa_\ga g_{\mu\nu}\sim 0,
\end{equation}
which expressed in a null frame becomes
\begin{equation}
-\frac{1}{2}\pa_\bL g_{L \ga} -\frac{1}{2} \pa_L g_{\bL \ga}+\pa_2 g_{2\ga}
+\pa_3 g_{3\ga}-\frac{1}{2} \pa_\ga \big(-g_{\bL L}+g_{22}+g_{33}\big)\sim 0
\end{equation}
Modulo tangential derivatives $\pa_L,\pa_2,\pa_3$ that we expect to be better
the wave coordinate condition reads
\begin{equation}
\pa_\bL g_{LL}\sim 0,\quad \pa_\bL g_{L2}\sim 0,\quad  \pa_\bL g_{L3}\sim 0,\quad \pa_\bL (g_{22}+g_{33})\sim 0
\end{equation}
which implies that
\begin{equation}
P(\pa_\bL g,\pa_\bL g)\sim -\frac{1}{2}
\big( (\pa_\bL g_{22})^2+(\pa_\bL g_{33})^2+2(\pa_\bL g_{23})^2\big)
\end{equation}
Consistent with this we choose
\begin{equation}
g_{22}=1+\chi_1(x_1-t),\quad g_{33}=1-\chi_1(x_1-t)
\end{equation}
and
\begin{equation}
g_{23}=g_{L2}=g_{L3}=0,
\end{equation}
These components solves the homogeneous wave equations \eqref{eq:redwave1}. In order to also solve the remaining wave equation \eqref{eq:redwave2} we must have
\begin{equation}
g_{\lb\lb}=-t\chi_2(x_1-t).
\end{equation}
In order to satisfy the remaining wave coordinate condition for $g_{\lb\lb}$ we
must have
\begin{equation}
\pa_L g_{\bl \bl}-2\de^{AB}\pa_A g_{B\bl}=0.
\end{equation}
To satisfy this we finally define
\begin{equation}
g_{B\bl}=-\frac14 x^B \chi_2(x_1-t),
\end{equation}
which also satisfy the wave equation \eqref{eq:redwave1}.

Based on the above linearized approximation we make the nonlinear ansatz in the table below
\begin{table}[ht]
\begin{tabular}{|l||c|c|c|c|}
\hline
%\diagbox{$Y=$}{$Z=$}
$g_{YZ}$

 & $L$ & $\bL$ & 2 & 3 \\[10pt]
\hline
\hline
$L$ &0 & -2 & 0 &0 \\[10pt]
\hline
$\bl$& -2 &  $-t\tilde{\chi}_2$ & $-\frac{1}{4}x_2(1+\chi_1)\tilde{\chi}_2$ & $-\frac{1}{4}(1+\chi_1)^{-1}x_3\tilde{\chi}_2$\\[10pt]
\hline
$2$ & 0 &$-\frac{1}{4}x_2(1+\chi_1)\tilde{\chi}_2$ & $1+\chi_1$ & 0\\[10pt]
\hline
3 & 0 &$-\frac{1}{4}x_3(1+\chi_1)^{-1}\tilde{\chi}_2$ & 0 & $(1+\chi_1)^{-1}$ \\[10pt]
\hline
\end{tabular}
\end{table}
with $\tilde{\chi}_2$ a modification of $\chi_2$:

This modification is obtained by trying to modify the metric above in order for it to satisfy the nonlinear wave coordinate condition. The reason this can be done is that
we first choose the metric so that $\det g=1$, in which the wave coordinate condition becomes a linear equation for the inverse of the metric
$$
\pa_\mu g^{\mu\nu}=\pa_L g^{L\nu}+\pa_{\bL} g^{\bL \nu}+\pa_1 g^{1\nu}+\pa_2 g^{2\nu}=0,
$$
solved in the same way we solved the linearized equation.

As it turns out with a metric in of the form in the table the only nonvanishing component of the curvature tensor is $R_{A\lb B\lb}\neq 0$ and  with $\tilde{\chi}_2$ satisfying
$$
\tilde{\chi}_2'-2(\chi'_1)^2(1+\chi_1)^{-2}-\tilde{\chi}_2^2/16=0,
$$
we have that the Ricci curvature $R_{\lb\lb}=g^{AB}R_{A\lb B\lb}=0$.

In the quasilinear case the domain has to be opened up slightly
away from the characteristic $t=x_1$, $x_2=x_3=0$,
to make sure the boundary of the domain is non-time
like and hence a domain of dependence. Since the metric is a small perturbation of the Minkowski metric in $L^\infty$ the light cones are close to those of Minkowski  and we only have to insure that the boundary is non time like.  Let $D$ be the domain
\begin{equation}
 D=\{(t,x);  (x_1-1)^2 H (x_1-1)+x_2^2/4+x_3^2/4< (1-t)^2\}
 \end{equation}
 where
 $H(x_1-1)=1$, when $x_1<1$ and $H(x_1-1)=1/4$, when $x_1>1$.
 The boundary consist of two parts $C=C_1\cup C_2$, where
\begin{equation}
 C_1=\{(t,x); \, x_1<1,\,\,\, (x_1-1)^2+x_2^2/4+x_3^2/4=(1-t)^2\},
\end{equation}
and
\begin{equation}
 C_2=\{(t,x); \, x_1\geq 1,\,\,\, (x_1-1)^2/4+x_2^2/4+x_3^2/4=(1-t)^2\}.
\end{equation}
$C_2$ is clearly non time like as is $C_1$ when $x_2^2+x_3^2\geq c>0$ since this
is true for the Minkowski metric with some room. In null coordinates
$u=(t-x_1)/2$, $v=(t+x_1)/2$, $C_1$ is given by
$$
4(1-v)u +x_2^2/4+x_3^2/4=0.
$$
The conormal is given by
$$
n=2(1-t)dt -2(1-x_1)dx_1 +x_2 dx_2/2+x_3 dx_3/2.
$$

Now its easy to see that the inverse of the metric takes the form

\begin{tabular}{|l||c|c|c|c|}
\hline
%\diagbox{$Y=$}{$Z=$}
%\diagbox{$Y=$}{$Z=$}
$g^{YZ}$ & $L$ & $\bL$ & 2 & 3 \\[10pt]
\hline
\hline
$L$ &$g^{LL}$ & $-\frac{1}{2}$ & $-\frac{1}{8}x_2\tilde{\chi}_2$ & $-\frac{1}{8}x_3\tilde{\chi}_2$ \\[10pt]
\hline
$\bl$& -$\frac{1}{2}$ &  0  & 0 &0\\[10pt]
\hline
$2$ & $-\frac{1}{8}x_2\tilde{\chi}_2$ &$0$  & $(1+\chi_1)^{-1}$ & 0\\[10pt]
\hline
3 & $-\frac{1}{8}x_3\tilde{\chi}_2$  &0 & 0 & $1+\chi_1$ \\[10pt]
\hline
\end{tabular}

It is easy to see that from this it follows that
$$
|g^{\alpha\beta} n_\alpha n_\beta -m^{\alpha\beta} n_\alpha n_\beta |
\lesssim (|\chi_1|+|\tilde{\chi}_2|)(x_2^2+x_3^2)+|g^{LL}| u^2,
$$
where $|u|\lesssim x_2^2+x_3^2$ on $C_1$,
and
$$
m^{\alpha\beta} n_\alpha n_\beta =-(x_2^2+x_3^2).
$$
Hence if $N$ is the normal to $C_1$ then
$$
g_{\alpha\beta}N^\alpha N^\beta\leq 0,
$$
so $C_1$ is non time like.

\section{The heuristic argument and Illposedness for the model system}
\label{sc:model}
 \subsection{The Reduced Einstein's Equations}
 Let $g$ be a solution of Einstein's equations
 \begin{equation}\label{Einstein}
 R_{\mu\nu}=0,
 \end{equation}
 in harmonic coordinates:
  \begin{equation}\label{eq:WaveCordinate}
\pa_\al (\sqrt{|g|}\, g^{\al\be})=0,\quad \be=0,\dots,3.
\end{equation}
  Denote the reduced wave operator by
 $$
 \Boxr =g^{\alpha\beta}\partial_\alpha\partial_\beta ,
 $$
 and let $h_{\a\b}\!=\!g_{\a\b}\!-\!m_{\a\b}$, and $m$ is the Minkowski metric. Then by \cite{LR} we have
\begin{equation}\label{eq:ReducedEinstein}
\Boxr_g h_{\mu\nu} =F_{\mu\nu} (h) (\pa h, \pa
h),
\end{equation}
where $F$ is a quadratic form in $\partial h$ with coefficients depending on $h$:
$$
F_{\mu\nu} (h) (\pa h, \pa h)=P(\partial_\mu h,\partial_\nu h)
+Q_{\mu\nu}(\pa h,\pa h) +G_{\mu\nu}(h)(\pa h,\pa h).
 $$
 Here
 $$
 P(h,k)
=\frac 14 h^{\a}_\a k^\b_\b-
\frac 12 h^{\a\b} k_{\a\b},
$$
where the indices are raised with respect to the Minkowski metric,
$Q_{\mu\nu}$ is a linear combinations of the standard
null-forms and $G_{\mu\nu}$  contains only cubic terms.
We want to construct a counter example to local existence in $H^2$. First by  \cite{KM} semilinear equations satisfying the classical nullcondition have local existence in $H^2$, so we can neglect these terms in a heuristic
argument. The counterexamples we construct below will be singular
along a light ray in such a way that $h$ vanishes exactly at the light cone
and therefore $|G_{\mu\nu}|\lesssim |h|\,|\partial h|^2$ will actually
be more regular than $|\partial h|^2$, so also this term can be neglected in
the heuristic argument. The counter example we construct will inside a light cone be a a function of $(t,x_1)$ only with a singularity along $t-x_1=0$,
but more regular in the $t+x_1$ direction and we therefore expect the derivatives in the $t-x_1$ direction to be worse than derivatives in the
other directions so expanding the metric in a null frame $L=\partial_t+\partial_1$, $\bL=\partial_t-\partial_1$, $A,B=\partial_2,\partial_3$, we see that $\partial_\mu$ is to leading order
$\frac{1}{2} L_\mu \pa_\bL$. We have $g^{\a\b}=m^{\a\b}-h^{\a\b}+O(h^2)$, where
$h^{\a\b}=m^{a\mu} m^{\b\nu} h_{\mu\nu}$ and  $h^{\alpha\beta}\partial_\alpha\partial_\beta$ is to leading order
$h_{LL}\partial_{\bL}^2$, where $h_{LL}=h_{\alpha\beta} L^\alpha L^\beta$.
Similarly $P(\partial_\mu h,\partial_\mu h) $ is to leading order given by
$ L_{\mu}L_{\nu} P(\pa_\bL h,\pa_\bL h)/4$. Hence expanding
$h$ in a null frame $h_{UV}=h_{\mu\nu} U^\mu U^\nu$, the reduced Einstein equations become to highest order
\begin{align}
\big(\Box -h_{LL}\pa_\bL^2) \,h_{TU}&\sim 0,\\
\big(\Box -h_{LL}\pa_\bL^2)\,  h_{\bL \bL}&\sim P(\pa_\bL h,\pa_\bL h),
\end{align}
where $T$ is any tangential frame component $T\in \{L,A,B\}$ and $U$ is any frame component $U\in\{\bL,L,A,B\}$. By \cite{LR}
\begin{multline}\label{eq:nullframeP}
P(p,k)
=\frac{1}{4}\delta^{AB}\big(2p_{A L}k_{B\underline{L}} +2p_{A
\underline{L}}k_{B{L}}- p_{AB}
k_{L\underline{L}}-p_{L\underline{L}}k_{AB} \big)\\
-\frac{1}{8}\big(p_{LL} k_{\underline{L}\underline{L}}
+p_{\underline{L}\underline{L}} k_{{L}{L}}\big)\\
-\frac{1}{4}\delta^{AB}\delta^{A^\prime
B^\prime}\big(2p_{AA^\prime}k_{BB^\prime}-
p_{AB} k_{A^\prime B^\prime}\big)
\end{multline}
The system simplifies further because as we shall see next the wave coordinate condition implies that
\begin{equation}
\pa_\bL h_{LL}\sim 0,\quad \pa_\bL h_{L2}\sim 0,\quad  \pa_\bL h_{L3}\sim 0,\quad \pa_\bL (h_{22}+h_{33})\sim 0
\end{equation}
which implies that
\begin{equation}
P(\pa_\bL h,\pa_\bL h)\sim -\frac{1}{2}\big( (\pa_\bL h_{22})^2+(\pa_\bL h_{33})^2+2(\pa_\bL h_{23})^2\big)
\end{equation}
and that after a possible change of variables we can also neglect the term
$h_{LL} \pa_\bL^2$.

\subsection{Illoposedness for the model problem}
Consider the following semilinear system:
\begin{equation} \Box \phi_2=-\big(\underline{L} \phi_1\big)^2,\qquad
\Box \phi_1=0,\qquad\text{where }
\underline{L}=\partial_t-\partial_{x^1}.
\end{equation}

Our first result using the techniques from \cite{L1} is illposedness for this system:

\begin{lemma}
\label{lm:chi1}
Let $\epsilon>0$ and set
$$
\chi_1(x_1)=\int_0^{x_1}{\epsilon|\log{|s/4|}|^\alpha \,ds},\quad 0<\alpha<1/2.
$$
There is $\Psi_1\in H^{2}(\Bbb R^3)$ such that
$$
\Psi_1(x)=\chi(x_1),\quad \text{in} \quad
B_0=\{x\in \Bbb R^3; (x_1-1)^2+x_2^2+x_3^2< 1\}
$$
and
$$
\|\Psi_1\|_{\dot H^{2}}\leq C_\alpha\epsilon,
$$
$\operatorname{supp}{\Psi_1}\subset\{x;|x|\leq 2\}$
and $\operatorname{singsupp}{\Psi_1}= \{0\}$.

Let
$$
\chi_2(x_1)=2\int_0^{x_1}\chi_1^\prime(s)^2 \,ds =2\int_0^{x_1}{\epsilon^2|\log{|s/4|}|^{2\alpha} \,ds},\quad 1/4<\alpha<1/2.
$$
There is $\Psi_{2t}\in \dot{H}^1(\Bbb R^3)$ such that
$$
\Psi_{2t}(x)=\chi_2(x_1-t)\quad\text{in}\quad B_t=\{x\in \Bbb R^3; (x_1-1)^2+x_2^2+x_3^2< (1-t)^2\}.
$$
For $0\leq t< 1$ we have
$$
\|\Psi_{2t}\|_{\dot{H}^2(B_t)}=\infty.
$$
\end{lemma}
\begin{proof} We have
%$$
%S_0(x_1)=\{(x_2,x_3)\in \Bbb R^2; \,(x_1,x_2,x_3)\in B_0\},
%$$
%and
%$$
%a_0(x_1)=\int_{S_0(x_1)}{dx_2 \,dx_3}=\int_{r^2\leq 2x_1-x_1^2}{2\pi rdr}=
%\pi (2x_1-x_1^2)\leq 2\pi x_1
%$$
%so
%$$
%\int_{B_0}{|\chi^{\prime\prime}(x_1)|^2\,dx}=
%\int_0^2{|\chi^{\prime\prime}(x_1)|^2 a_0(x_1)\,dx_1}\leq
%\int_0^2{\frac{2 \epsilon^2\pi\alpha^2\,x_1 dx_1}{ x_1^2  |\log{|x_1/4|}|^{2(1-\alpha)}  } }
%<\infty .
%$$
\begin{multline*}
\int_{B_0}{|\chi_1^{\prime\prime}(x_1)|^2\,dx}=
\int_0^2{|\chi_1^{\prime\prime}(x_1)|^2
    \bigg(\int_{x_2^2+x_3^2\leq 2x_1-x_1^2}{dx_2 dx_3}\bigg)\,dx_1}\\
=\int_0^2{|\chi_1^{\prime\prime}(x_1)|^2 \pi (2x_1-x_1^2)\,dx_1}\leq
\int_0^2{\frac{2 \epsilon^2\pi\alpha^2\, x_1 dx_1}{  x_1^2  |\log{|x_1/4|}|^{2(1-\alpha)}  } }
<\infty .
\end{multline*}
Hence $\|\Psi_1\|_{\dot{H}^2(B_0)}\leq C_\alpha \epsilon$ and
it follows from extension theorems in
Stein\cite{S}, (see page 181)
that it can be extended to a function in $H^{2}(\Bbb R^3)$ with comparable
norm. Moreover, the extension can be chosen to satisfy the above support and
singular support properties.

Moreover, if $0\leq t<1$,
\begin{multline*}
\int_{B_t}{|\chi_2^{\prime\prime}(x_1-t)|^2\,dx}\\
=\int_t^{2-t}{|\chi_2^{\prime\prime}(x_1-t)|^2
    \bigg(\int_{x_2^2+x_3^2\leq (2-(x_1+t))(x_1-t)}{dx_2 dx_3}\bigg)\,dx_1} \\
=\int_t^{2-t}{|\chi_2^{\prime\prime}(x_1-t)|^2 \pi (2-(x_1+t))(x_1-t)\,dx_1}
\\
\geq
2\int_0^{1-t}{\frac{\epsilon^4(1-t)\pi(2\alpha)^2\, x_1 dx_1}{  x_1^2  |\log{|x_1/4|}|^{2(1-2\alpha)}  } }
=\infty .
\end{multline*}
\end{proof}

The data we will choose for (1.1) are
\begin{equation*}
\phi_1(0,x)=\Psi_1(x),\,\quad\quad\quad \partial_t \phi_1(0,x)=-\partial_{x_1}\Psi_{1}(x).
\end{equation*}
Note now, that by a domain of dependency argument the solution of (1.1)
inside the cone $\Lambda=\{(t,x);|x-(1,0,0)|\leq 1-t, t\geq 0\}$,
only depend on the data inside the ball $B_0$.
Since data inside the ball $B_0$ only depends on $x_1$, the solution $\phi_1$
inside $\Lambda$ satisfy
$$
\paxmt\patpx \phi_1(t,x_1)=0.
$$
It follows that $\patpx \phi_1=0$ in $\Lambda$ and hence
$$
\phi_1(t,x_1)=\chi_1(x_1-t),\qquad (t,x)\in \Lambda.
$$
Hence
$$
\paxmt\patpx \phi_2(t,x_1)=-\Big(\paxmt\phi_1(t,x_1)\Big)^2=
-4\chi_1^\prime(x_1-t)^2.
$$
We now choose data
\begin{equation*}
\phi_2(0,x)=0,\,\quad\quad\quad \partial_t \phi_2(0,x)=-\Psi_{20}(x).
\end{equation*}
It then follows that in $\Lambda$
$$
\phi_2(t,x)=-t\chi_2(x_1-t).
$$
Hence by the estimate in the lemma
$$
\|\phi_2(t,\cdot)\|_{H^2}=\infty,\qquad\text{if}\quad 0<t<1.
$$
On the other hand it easily follows from standard Strichartz estimates that
$$
\|\phi_2(t,\cdot)\|_{H^{2-\delta}}<\infty,\qquad\text{if}\quad \delta>0.
$$

\subsection{The wave coordinate condition}

We prefer to work with lower indices since the nonlinearity is more transparent in this case.  We collect two standard linear algebra results about the derivative of the determinant of a matrix and the inverse of a matrix:
\begin{lemma} Let $|g|=|\det g\,|$.
We have
\begin{align}
\pa_\al |g|&=|g|\,g^{\mu\nu}\, \pa_\al g_{\mu\nu}, \label{eq:det}\\
\pa_\al g^{\mu\nu}&=-g^{\mu\mu_1}\,g^{\nu\nu_1}\,\pa_\al g_{\nu_1\mu_1}\label{eq:inv}.
\end{align}
\end{lemma}

We convert the constraint equations
\[
\pa_\al (\sqrt{|g|}\, g^{\al\be})=0,\quad \be=0,\dots,3,
\]
using the Lemma above.
We get
\[
-\sqrt{|g|}\,g^{\al\al_1}g^{\be\be_1}\pa_\al g_{\al_1\be_1}+\frac{1}{2} g^{\al\be}\, \sqrt{|g|} \,g^{\mu\nu}\pa_\al g_{\mu\nu}=0.
\]
Apply $g_{\be\ga}$, divide by $\sqrt{|g|}$ and relabel the indices ($\al\rar \mu$, $\al_1 \rar \nu$),  to arrive at:
\bb
\label{eq:conslow}
-g^{\mu\nu}\pa_\mu g_{\nu\ga}+\frac{1}{2}g^{\mu\nu}\,\pa_\ga g_{\mu\nu}=0,
\ee
which is the form that we will use.

Write down the linearization of the wave coordinate condition (\ref{eq:conslow})
that for small $h$ is good approximation of the wave coordinate condition:
\bb
\label{eq:consep}
-m^{\mu\nu}\pa_{\mu}h_{\nu\ga}+\frac{1}{2}m^{\mu\nu}\, \pa_{\ga} h_{\mu\nu}=0.
\ee
Define the basis (null frame) by
\bb
\label{eq:nullfr}
L=\pa_t+\pa_{x_1},\quad \bl=\pa_t-\pa_{x_1},\quad\pa_A=\pa_{x_A},\quad A=2,3.
\ee
We use the basis from (\ref{eq:nullfr}) in (\ref{eq:consep}).
We have for $\ga=L$:
\[
\frac12 \pa_{\bl}h_{LL}+\frac12 \pa_L h_{\bl L}-\de^{AB}\pa_A h_{BL}+\frac12 \pa_L \left(-h_{L\bl}+\de^{AB}h_{AB} \right)=0,
\]
for $\ga=\bl$:
\[
\frac12 \pa_{\bl} h_{L \bl}+\frac12 \pa _L h_{\bl \bl}-\de^{AB}\pa_A h_{B\bl}+\frac12 \pa_\bl \left(-h_{L\bl}+\de^{AB}h_{AB} \right)=0,
\]
for $\ga=C\in \{2,3\}$:
\[
\frac12 \pa_{\bl} h_{L C}+\frac12 \pa _L h_{\bl C}-\de^{AB}\pa_A h_{BC}+\frac12 \pa_C \left(-h_{L\bl}+\de^{AB}h_{AB} \right)=0.
\]
In the first two equations, $h_{L\bl}$ coefficient cancels and therefore, we can write the linearized wave coordinate condition as follows:
\bse
\label{eq:Consh1}
\begin{align}
\pa_{\bl}h_{LL}-2\de^{AB}\pa_A h_{BL}+ \pa_L\left( \de^{AB}h_{AB}\right)&=0,
\label{eq:ConsL}
\\
\pa _L h_{\bl \bl}-2\de^{AB}\pa_A h_{B\bl}+\pa_\bl\left( \de^{AB}h_{AB}\right)&=0,
\label{eq:ConsbL}
\\
 \pa_{\bl} h_{L C}+ \pa _L h_{\bl C}-2\de^{AB}\pa_A h_{BC}+ \pa_C \left(-h_{L\bl}+\de^{AB}h_{AB} \right)&=0\label{eq:ConsC}.
\end{align}
\ese
Recall that our solution is
\begin{align*}
h_{AB}&\sim \phi_1,\\
h_{\bl\,\bl}&\sim \phi_2,
\end{align*}
where $\phi_1\sim \chi_1(x_1-t)$ and $\phi_2\sim -t\chi_2(x_1-t)$ inside the cone $|x|\leq 1-t$ with $\phi_1\in H^2$ while $\phi_2\in H^{2-\de}\setminus H^2$.
\subsubsection{Eliminating truly bad parts}
We would like to eliminate the components that are differentiated by $\bl$ in (\ref{eq:ConsL})-(\ref{eq:ConsC}) as they wouldn't have the same regularity as derivatives of $L,A$. Therefore, identifying these terms in (\ref{eq:ConsL})-(\ref{eq:ConsC}), respectively, we set
\begin{align}
h_{LL}&=0,\label{eq:hLL}
\\
\de^{AB}h_{AB}&=h_{22}+h_{33}=0, \label{eq:notrace} \\
\label{eq:hLC} h_{LC}&=0.
\end{align}
We can't set $h_{AB}=0$ but it is enough to have
\bb
\begin{split}
\label{eq:hprime}
h_{22}&=-h_{33}=\phi_1,\\
h_{23}&=0,
\end{split}
\ee
\subsubsection{Satisfying the first linearized wave coordinate condition (\ref{eq:ConsL})}
With (\ref{eq:hLL}),(\ref{eq:notrace}),(\ref{eq:hLC}), the first constraint is satisfied automatically.
\subsubsection{Satisfying the second linearized wave coordinate condition (\ref{eq:ConsbL})}
With the choice $h_{AB}$ as in (\ref{eq:hprime}), the constraint (\ref{eq:ConsbL}) becomes
\[
\pa_L h_{\bl \bl}-2\de^{AB}\pa_A h_{B\bl}=0.
\]
Since $\pa_L h_{\bl\bl}=-t\chi_2$ inside the cone, this suggests to define
\bb
\label{eq:hblC}
h_{B\bl}=-\frac14 x^B \chi_2.
\ee
Observe that $x^B \phi_2\in H^2$ as near the singular point $x_1=t$ of $\phi_2$, inside the cone, we have $|x^B|\ls (t-x_1)^\frac12$, which makes the appropriate expression integrable and prevents the singularity.
\subsubsection{Satisfying the third linearized wave coordinate condition (\ref{eq:ConsC})}
We have
\[
\pa_A h_{BC}=0,
\]
by (\ref{eq:hprime}). Also,
\[
\pa_L h_{B\bl}=0
\]
by (\ref{eq:hblC}). Combining this with (\ref{eq:hLL}),(\ref{eq:hLC}), we see that the last constraint (\ref{eq:ConsC}) is reduced to
\[
\pa_C h_{L\bl}=0,
\]
which suggests
\[
h_{L\bl}=0.
\]
To summarize, in the $L,\bl,\pa_2,\pa_3$ basis and in that order, $h_{\al\be}$ is
\[
h_{YZ}|_{\{|x|<1-t\}}=\begin{pmatrix}
0 & 0 & 0 & 0 \\ 0 & -t\,\chi_2 & -\frac{1}{4}x_2\, \chi_2 & -\frac{1}{4} x_3\chi_2 \\
0 &-\frac{1}{4}x_2\,\chi_2 & \chi_1 & 0 \\ 0 & -\frac{1}{4} x_3 \chi_2 & 0 & -\chi_1
\end{pmatrix}.
\]

\section{The solution inside the cone}
The goal of this section is to build on the ideas of Section \ref{sc:model} to obtain a solution of the Einstein equations inside the cone and wave coordinates for it, such that the metric in these coordinates has a finite $H^2$ norm at time zero and infinite at all other times. For this end,
let $D$ be the domain
\begin{equation}
 D=\{(t,x);  (x_1-1)^2 H (x_1-1)+x_2^2/4+x_3^2/4< (1-t)^2\},
 \end{equation}
 where
 $H(x_1-1)=1$, when $x_1<1$ and $H(x_1-1)=1/4$, when $x_1>1$.
Set $D_t=\{x;\, (t,x)\in D\}$.
Our goal is to prove the following statement.
\begin{theorem} \label{th:inside}
%There exists $\de>0$, such that t
There exist a space time $(D, g)$ and coordinates $x_{\al}:D\rar \R,\al=0,1,2,3$ such that
\begin{itemize}
\item The metric $g$ satisfies the Einstein vacuum equation
\[
\text{Ric}(g)=0,\quad\text{ on } D.
\]
\item The coordinates $x_{\al}$ are wave coordinates
\[
\pa_\al (\gv g^{\al\be})=0,\quad \text{ on } D\quad \be=0,1,2,3.
\]
\item The metric $g$ has finite initial data in $H^2(D_0)\times H^1(D_0)$:
\[
\|g_{\al\be}\|_{H^2(D_0)}+\|\pa_t g_{\al\be}\|_{H^1(D_0)}<\infty,\quad \al,\be\in \{0,1,2,3\}.
\]
\item The $H^2(D_{t})$ norm of $g_{00}$ at any other time $t$ is infinite:
\[
\|g_{00}\|_{H^2(D_{t})}=\infty,\quad \forall t\in (-1,1)\setminus \{0\}.
\]
\end{itemize}
\end{theorem}
We prove the theorem by describing an explicit example for such a metric $g$ and coordinates $x_\al$. The coordinates $x_{\al}$ are  the standard coordinates on $\R^{1+3}$
\[
x_{\al}((y_0,y_1,y_2,y_3))=\de_{\al}^\be y_\be.
\]
We will also write $t=x_0$ We use the rest of this section to specify $g$ and verify that the hypotheses of Theorem \ref{th:inside} are satisfied. We define the following vector fields:
\bb
\begin{split} L&=\pa_t+\pa_{x_1},\\
\bL&=\pa_t-\pa_{x_1}.
\end{split}
\ee
We complete $\{L,\bl\}$ to a basis by adding $\pa_A=\pa_{x_A},A=2,3$. In what follows we will use $A,B$ to denote an index from a set $\{2,3\}$. Since $L,\bl$ are constant coefficient vector fields, we will abuse the notation and treat $L,\bl$ as fictitious indices as well. For example $\pa_{L}f=\pa_{t}f+\pa_{x_1}f$ or $g_{L\bL}=\la L, \bl \ra_{g}$.
\begin{rema} Since  $\{ L,\bl,\pa_A| A=2,3\}$ forms a basis and have constant coefficient we use this basis instead of the standard one in all subsequent derivations.
\end{rema}
We can now specify the metric.
\begin{defin}
\label{df:metric}
 The nonzero coefficients of  the metric $g$ in the basis $\{L,\bL,\pa_2,\pa_3\}$ are as follows
\begin{align}
g_{L\bL}&=-2\\
g_{\bL\bL}&=-t \tilde{\chi}_2(x_1-t),\\
g_{AB}&=\de_{AB}\chi_{1A}(x_1-t),\\
g_{A\bL}&=-\frac{1}{4}x_A \chi_{2A}(x_1-t),
\end{align}
where
%\begin{align*}
%\end{align*}
\[
\chi_{12}=1+\chi_1=\frac{1}{\chi_{13}}.
\]
Here
$\chi_1$ was defined in Lemma \ref{lm:chi1}, $\tilde{\chi}_2$ is a slight modification of $\chi_2$ that will be defined in Lemma \ref{lm:Ric} below and
\[
\chi_{2A}=\chi_{1A}\tilde{\chi}_2.
\]
The rest of the coefficients are given by symmetry.
\end{defin}
\begin{rema} Unless we specify otherwise, the argument of the $\chi$-functions will be $x_1-t$.
\end{rema}
The coefficients of $g$ are summarized in Table \ref{tb:gdown} below.
%\begin{table}
%\begin{tabular}{|l|ccc|}
%\hline
%\diagbox{Time}{Day} & Mon & Tue & Wed \\
%\hline
%Morning & used & used & \\
%Afternoon & & used & used \\
%\hline
%\end{tabular}
\begin{table}[ht]
\begin{tabular}{|l||c|c|c|c|}
\hline
%\diagbox{$Y=$}{$Z=$}
$g_{YZ}$

 & $L$ & $\bL$ & 2 & 3 \\[10pt]
\hline
\hline
$L$ &0 & -2 & 0 &0 \\[10pt]
\hline
$\bl$& -2 &  $-t\tilde{\chi}_2$ & $-\frac{1}{4}x_2(1+\chi_1)\tilde{\chi}_2$ & $-\frac{1}{4}(1+\chi_1)^{-1}x_3\tilde{\chi}_2$\\[10pt]
\hline
$2$ & 0 &$-\frac{1}{4}x_2(1+\chi_1)\tilde{\chi}_2$ & $1+\chi_1$ & 0\\[10pt]
\hline
3 & 0 &$-\frac{1}{4}x_3(1+\chi_1)^{-1}\tilde{\chi}_2$ & 0 & $(1+\chi_1)^{-1}$ \\[10pt]
\hline
\end{tabular}
\caption{The coefficients $g_{YZ}$ of the metric}
\label{tb:gdown}
\end{table}
Thus we reduce the proof of Theorem \ref{th:inside} to the following three Lemmas
\begin{lemma}
\label{lm:Ric}
Let $\tilde{\chi}_2$ satisfy
\[
\tilde{\chi}_2'-2(\chi'_1)^2(1+\chi_1)^{-2}-\frac{1}{16}\tilde{\chi}_2^2=0,
\]
then the metric $g$ defined in Definition \ref{df:metric} satisfies
\[
\text{Ric}(g)=0.
\]
\end{lemma}
\begin{lemma}
\label{lm:wave}
Let $g$ be the metric defined in Definition \ref{df:metric} then the standard coordinates satisfy the wave coordinate condition (\ref{eq:WaveCordinate}).
\end{lemma}
\begin{lemma}
\label{lm:norms}The metric $g$ in Definition \ref{df:metric} satisfies
\begin{align*}
\|g_{\al\be}\|_{H^2(D_0)}+\|\pa_t g_{\al\be}\|_{H^1(D_0)}&<\infty,\quad \al,\be\in \{0,1,2,3\},\\
\|g_{00}\|_{H^2(D_{\de'})}&=\infty,\quad \forall \de'\in (-\de,\de)\setminus \{0\}.
\end{align*}
\end{lemma}
The Lemmas \ref{lm:Ric}, \ref{lm:wave} are given by direct computation for which we will provide some intermediate steps. The following statement is a straightforward observation
\begin{claim}
We have the following equalities
\begin{enumerate}
\label{cl:linalg}
\item
\[
\gv=2.
\]
\item The non-zero coefficients of the inverse metric $g^{YZ}$ are as follows:
\begin{align*}
g^{LL}&=\frac{1}{4}t\tilde{\chi}_2+\frac{1}{64}\Big( x_2^2(1+\chi_1)+x_3^2(1+\chi_1)^{-1}\Big)\tilde{\chi}_2^2,\\
g^{LA}&=-\frac{1}{8}x_A \chi_{2A}(\chi_{1A})^{-1},\\
g^{L\bL}&=-\frac{1}{2}\\
g^{AB}&=\de_{AB}(\chi_{1A})^{-1},\\
\end{align*}
and their symmetric counterparts.
\end{enumerate}
\end{claim}
We summarize $g^{YZ}$ in Table \ref{tb:gup} below. \begin{table}[ht]
\begin{tabular}{|l||c|c|c|c|}
\hline
%\diagbox{$Y=$}{$Z=$}
%\diagbox{$Y=$}{$Z=$}
$g^{YZ}$ & $L$ & $\bL$ & 2 & 3 \\[10pt]
\hline
\hline
$L$ &$\frac{1}{4}t\tilde{\chi}_2+\frac{1}{64}\big( x_2^2(1+\chi_1)+x_3^2(1+\chi_1)^{-1}\big)\tilde{\chi}_2^2$ & $-\frac{1}{2}$ & $-\frac{1}{8}x_2\tilde{\chi}_2$ & $-\frac{1}{8}x_3\tilde{\chi}_2$ \\[10pt]
\hline
$\bl$& -$\frac{1}{2}$ &  0  & 0 &0\\[10pt]
\hline
$2$ & $-\frac{1}{8}x_2\tilde{\chi}_2$ &$0$  & $(1+\chi_1)^{-1}$ & 0\\[10pt]
\hline
3 & $-\frac{1}{8}x_3\tilde{\chi}_2$  &0 & 0 & $1+\chi_1$ \\[10pt]
\hline
\end{tabular}
\caption{The coefficients $g^{YZ}$ of the metric.}
\label{tb:gup}
\end{table}
\begin{proof}[Proof of Lemma \ref{lm:Ric}]
We use the slightly nonstandard definition of Christoffel symbols from \cite[(3.1)]{LR}:
%\begin{align*}
\bb
\label{eq:Chr}
\begin{split}
\G{\al}{\be}{\ga}&=\frac{1}{2}g^{\be\de}\left( \pa_\al g_{\de\ga}+\pa_{\ga}g_{\de\al}-\pa_{\de}g_{\al\ga}\right),\\
\Ga_{\al\be\ga}&=g_{\be\de}\G{\al}{\de}{\ga}=\frac{1}{2}\left( \pa_\al g_{\be\ga}+\pa_{\ga}g_{\be\al}-\pa_{\be}g_{\al\ga}\right).
\end{split}
\ee
The following Christoffel symbols are not zero:
%The following ones are not zero
\begin{align*}
\Ga_{L\bl\bl}=\Ga_{\bl\bl L}=-\Ga_{\bl L \bl}
=&-\frac{\tilde{\chi_2}}{2},\\
\Ga_{\bl\bl\bl}
=&-\frac{\tilde{\chi}_2}{2}+t\tilde{\chi}_2',\\
\Ga_{\bl A\bl}
=&\frac{1}{2}x_a\chi_{2A}^\prime,\\
\Ga_{\bl BA}=\Ga_{AB \bl}=&-\de_{AB}\chi_{1A}',\\
\Ga_{A\bl B}
=&\de_{AB}\left(-\frac{1}{4}\chi_{2A}+\chi_{1A}'\right).
\end{align*}
whereas $\Gamma_{LLL}=\Gamma_{L\bl L}=\Gamma_{LL\bl}=0$ and
$$
\Gamma_{A\bl\bl}=\Gamma_{A L L}=\Gamma_{L A L}=\Gamma_{A L \bl}=\Gamma_{L A \bl}=\Gamma_{A L B}=\Gamma_{ABL}=\Gamma_{ABC}=0
$$
With the convention (\ref{eq:Chr}), we have the following formula for the curvature (see also \cite[3.10]{LR}):
\[
R_{\mu\al\nu\be}=\pa_\be \Ga_{\mu\al\nu}-\pa_\nu \Ga_{\mu\al\be}+\Ga_{\nu\lm\al}\G{\mu}{\lm}{\be}-\Ga_{\al\lm\be}\G{\mu}{\lm}{\nu}.
\]
We will split the curvatures into two non-tensors, which represent the linear and the quadratic parts
\begin{align*}
R_{\mu\al\nu\be}&=\rl_{\mu\al\nu\be}+\rq_{\mu\al\nu\be},\\
\rl_{\mu\al\nu\be}&=\pa_\be \Ga_{\mu\al\nu}-\pa_\nu \Ga_{\mu\al\be},\\
\rq_{\mu\al\nu\be}&=g^{\lambda\gamma}\Ga_{\nu\lm\al}\Ga_{\mu\gamma\be}
-g^{\lambda\gamma}\Ga_{\al\lm\be}\Ga_{\mu\gamma\nu}.
\end{align*}
%The following components are zero
%\begin{align*}
%\rl_{\mu L \nu\be}&=\rq_{\mu L \nu\be}=0\\
%\rl_{A\mu BC}&=\rq_{A\mu BC}=0.
%\end{align*}
%This leaves us with one component, up to symmetries,  of the curvature, which are:
We claim that the only nontrivial components of the non-tensors $\rl,\rq$ are
$\rl_{A\bl B \bl}$ and $\rq_{A\bl B\bl}$. In fact this follows from the symmetries if we can show that they vanish if at least one index is $L$ or at least three index are $A,B,C$.
For $\rl_{A\bl B \bl}$ the first follows since the only components of $\pa_\beta \Gamma_{\mu\alpha\nu}$ with at least one $L$ are $\pa_{\bl} \Gamma_{\bl L\bl}$, $\pa_\bl \Gamma_{\bl \bl L}=\pa_\bl \Gamma_{L \bl \bl}$ and $\pa_L \Gamma_{\bl \bl\bl}$ but they are seen to cancell each other when appearing in $\rl$. Secondly $\Gamma_{ABC}=0$ and
$\pa_A\Gamma_{BC\bl}=\pa_A \Gamma_{\bl CB}=0$ and $\pa_A\Gamma_{B \bl C}=0$ which concludes the proof of the statement for $\rl$. For $\rq$ the first follows since the only combination of $g^{\lambda\gamma} \Gamma_{\nu\lambda\alpha} \Gamma_{\mu\gamma\beta} $ with one
index $L$ say $\nu=L$, is $g^{\bl L} \Gamma_{L\bl\bl} \Gamma_{ \bl L \bl}$ which
will cancell when appearing in $\rq$. Secondly if three of the indeces of
$g^{\lambda\gamma} \Gamma_{\nu\lambda\alpha} \Gamma_{\mu\gamma\beta} $ are $A,B,C$ say
$\nu=A$ and $\alpha=B$ and $\beta=C$ then in fact $g^{\bl L} \Gamma_{A\bl B} \Gamma_{\mu L C}=0$ which concludes the proof of the claim.

We have
\begin{align*}
\rl_{A\bl B \bl}&=\de_{AB} \left(\frac{1}{2}\chi'_{2A}-2\chi''_{1A}\right),\\
\rq_{A\bl B\bl}&=\de_{AB}\left[\chi_{1A}^{-1} (\chi_{1A}')^2+\frac{1}{4}\tilde{\chi}_2( -\frac{1}{4}\chi_{2A}+\chi'_{1A})\right]
\end{align*}
The last follows since
\begin{multline}
\rq_{A\bl B\bl}=\Gamma_{B C\bl} \Gamma_{A\,\,\,\bl}^{\,\,\, C}-\Gamma_{\bl L \bl}
\Gamma_{A\,\,\,B}^{\,\,\, L}-\Gamma_{\bl \bl \bl}\Gamma_{A\,\,\,B}^{\,\,\, \bl}\\
=g^{CD} \Gamma_{B C\bl} \Gamma_{A D\bl}-g^{L\bl}\Gamma_{\bl L \bl}
\Gamma_{A\bl B}
\end{multline}
With this we compute \footnote{We assume the summation convention on $A$,$B$.}:
\begin{align*}
g^{AB}\rl_{A\bl B \bl}&=\tilde{\chi}_2^\prime-4 \frac{(\chi_1')^2}{(1+\chi_1)^{2 }},\\
g^{AB}\rq_{A\bl B \bl}&=(\chi_1')^2\left[ \frac{2}{(1+\chi_1)^2}\right]-\frac{1}{16}\tilde{\chi}_2^2.
\end{align*}
%\ee
We use this to compute the only non-zero component of the Ricci curvature - $\text{Ric}_{\bl\bl}$:
\begin{align*}
\text{Ric}_{\bl\bl}&=g^{AB}\rl_{A\bl B \bl}+g^{AB}\rq_{A\bl B \bl}\\
&=\tilde{\chi}^\prime_2-2 \frac{(\chi_1')^2}{(1+\chi_1)^{2 }}-\frac{1}{16}\tilde{\chi}_2^2.
\end{align*}
\end{proof}
\begin{proof}[Proof of Lemma \ref{lm:wave}] Since $\gv$ is constant by item 1 of Claim \ref{cl:linalg}, we use some elementary linear algebra to rewrite the wave coordinate condition (\ref{eq:WaveCordinate}) as
\bb
\label{eq:wave2}
g^{\mu\nu}\pa_\mu g_{\nu\ga}=0, \quad \ga=0,..,3.
\ee
Denote
\[
d_\ga=g^{\mu\nu}\pa_\mu g_{\mu\ga}.
\]
Our goal is to show $d_\ga=0$ for $\ga=0,..3$. Instead, we will show $d_L=d_\bl=d_A=0$ for $A=2,3$, which is equivalent since $L,\Bl, A$ form a basis of constant coefficient vector fields.
%Since the $L,\Bl, A$ are constant coefficient vector fields, we will again use them as a basis ins
%We will use $L,\bL,A$ vector fields since they are a basis for constant coefficient vector fields. We need to verify (\ref{eq:wave2}) for $\ga=L,\Bl,A$. This is trivial for $\ga=L$ since all the coefficients of type $g_{UL}$ are constant. For
The fact that $d_L=0$ is obvious, since the metric coefficients of the form $g_{XL}$ are constant. For $d_\bl$, we write
\begin{align*}
d_\bl=&g^{LL}\pa_L g_{L\bl}+g^{L\bl}\pa_L g_{\bl\bl}+g^{LA}\pa_{L}g_{A\bl}&\\
&+g^{\bl L}\pa_{\bl}g_{L\bl}+g^{AL}\pa_A g_{L\bL}+g^{AB}\pa_A g_{B\bl}
\end{align*}
Since coefficients $g_{XL}$ are constant, we drop their derivatives
\[
d_\bl=g^{L\bl}\pa_L g_{\bl\bl}+g^{LA}\pa_{L}g_{A\bl}+g^{AB}\pa_A g_{B\bl}.\\
\]
Observe that $\pa_L g_{A\bl}=\frac{1}{4}\pa_L (x_A \chi_{2A}(x_1-t))=0.$ Therefore
\begin{align*}
d_\bl&=g^{L\bl}\pa_L g_{\bl\bl}+g^{AB}\pa_A g_{B\bl}\\
&=g^{L\bl}\pa_L g_{\bl\bl}+g^{22}\pa_2 g_{2\bl}+g^{33}\pa_3 g_{3\bl}\\
&=-\frac{1}{2}\pa_L(-t\tilde{\chi}_2)-\frac{1}{4}(1+\chi_1)^{-1}\pa_2(x_2(1+\chi_1)\tilde{\chi}_2)\\
&\quad - \frac{1}{4}(1+\chi_1)\pa_3(x_3(1+\chi_1)^{-1}\tilde{\chi}_2)\\
&=0,
\end{align*}
since $\pa_L (t\tilde{\chi}_2(x_1-t))=\tilde{\chi}_2(x_1-t)$. Lastly
\begin{align*}
d_A=&g^{LL}\pa_L g_{LA}+g^{L\bl}\pa_L g_{\bl A}+g^{LB}\pa_{L}g_{AB}\\
&+g^{\bl L}\pa_{\bl}g_{LA}+g^{BL}\pa_B g_{LA}+g^{BC}\pa_B g_{CA}.
\end{align*}
We drop derivatives of the constant coefficients $g_{XL}$
\[
d_A=g^{L\bl}\pa_L g_{\bl A}+g^{LB}\pa_{L}g_{AB}+g^{BC}\pa_B g_{CA}.
\]
Next, observe that $g_{\bl A}$ depends only on $x_1-t$ and $x_A$, thus $\pa_L g_{\bl A}=0$. Similarly $g_{AB}$ depends only on $x_1-t$, therefore $\pa_L g_{AB}=0$. Similarly, $g_{CA}$ depends only on $t-x_1$ and therefore $\pa_B g_{CA}=0$. Thus we arrive at the conclusion
\[
d_A=0,
\]
which completes the proof.
\end{proof}
%\begin{rema}
%\begin{itemize}
%\item Our metric $g$ is mostly one dimensional and depends on $t-x_1$ with some deviations needed to account for the wave coordinate condition.
%\item
%\end{itemize}
%\end{rema}
\begin{proof}[Proof of Lemma \ref{lm:norms}] The function $\chi_1$ has been analyzed in Lemma \ref{lm:chi1}. Thus, to prove the lemma, it is enough to establish the following:
\begin{align*}
\tilde{\chi}_2&\in H^1(D_t)\setminus H^2(D_t),\quad t\in[0,1]\\
x^A\tilde{\chi}_2&\in H^2(D_0),\quad A=2,3.
\end{align*}
Recall that $\tilde{\chi}_2$ satisfies
\bb
\label{eq:chi2ode}
\diff{}{x} \tilde{\chi}_2=2(\chi_1'(x))^2(1+\chi_1(x))^{-2}+\frac{1}{16}\tilde{\chi}_2(x)^2.
\ee
We choose $\tilde{\chi}_2(0)=0$ then by integrating (\ref{eq:chi2ode}), we can show that $\tilde{\chi}_2$ is bounded by 2 for $|y|\leq 1$ if we adjust $\epsilon$ in the definition of $\chi_1$, so that $\int_0^1 (\chi'_1)^2\leq 1$ and apply the bootstrap assumption $\tilde{\chi}_2\leq 4$ for the integral $\int_0^1\tilde{\chi}_2^2$. The same argument works to show that $\tilde{\chi}_2'\in L^2(D_0)$. To show that $\tilde{\chi}_2''\notin L^2$, we differentiate (\ref{eq:chi2ode}). We will have
\bb
\label{eq:chi2odediff}
\tilde{\chi}_2''=2(\chi_1'(y))^2(1+\chi_1^2)^{-2}+F(\chi_1,\chi_1',\tilde{\chi}_2,\tilde{\chi}_2'),
\ee
where $F$ will have a smooth dependance on $\chi_1,\tilde{\chi}_2$ and polynomial in $\chi_1',\tilde{\chi}_2'$. Since $\chi_1',\tilde{\chi}_2'\in L^p$ for any $p<\infty$, we conclude that
\[
\|F(\chi_1,\chi_1',\tilde{\chi}_2,\tilde{\chi}_2')\|_{L^2}\leq C<\infty
\]
Also since $\chi_1$ is bounded, we can bound $(1+\chi_1(y))^{-2}\geq c>0$. Therefore, applying the same logic as in Lemma \ref{lm:chi1}, we will arrive at
\begin{equation*}
\int_{D_t} \tilde{\chi}_2''(x_1-t)^2\geq  c\int_0^{1-t}{\frac{\epsilon^4(1-t)\pi(2\alpha)^2\, x_1 dx_1}{  x_1^2  |\log{|x_1/4|}|^{2(1-2\alpha)}  } }-\frac{1}{2}C^2\\
=\infty.
\end{equation*}
Thus it remains to show that $x_A \tilde{\chi}_2\in H^2(B_0)$. Without loss of generality, put $A=2$. The only estimate which is not addressed above is $x_2\tilde{\chi}''_2\in L^2$, since we have already shown $\der{}{x_2}(x_2 \tilde{\chi}_2)\in H^1$. We use (\ref{eq:chi2odediff}) to obtain the following estimate
\begin{multline*}
\int_{D_t}{x_2^2|\tilde{\chi}_2^{\prime\prime}(x_1-t)|^2\,dx}\\
\ls\int_t^{2-t}{|\tilde{\chi}_2^{\prime\prime}(x_1-t)|^2
    \bigg(\int_{x_2^2+x_3^2\leq (2-(x_1+t))(x_1-t)}{x_2^2\, dx_2 dx_3}\bigg)\,dx_1} \\
\ls\int_t^{2-t}{|\tilde{\chi}_2^{\prime\prime}(x_1-t)|^2 \pi (2-(x_1+t))^2(x_1-t)^2\,dx_1}
\\
\ls
2\int_0^{1-t}{\frac{\epsilon^4(1-t)^2\pi(2\alpha)^2\, dx_1}{   |\log{|x_1/4|}|^{2(1-2\alpha)}  } }
<\infty,
\end{multline*}
which concludes the proof of the Lemma.
%Integrating the differential equation above
\end{proof}

\section{Taking into account the bending of the light cones}
To take into account the bending of the light cones in the metric we need
to open up our domain slight to ensure its space like or null. We will therefore replace our domain.
 Let $D$ be the domain
\begin{equation}
 D=\{(t,x);  (x_1-1)^2 H (x_1-1)+x_2^2/4+x_3^2/4< (1-t)^2\}
 \end{equation}
 where
 $H(x_1-1)=1$, when $x_1<1$ and $H(x_1-1)=1/4$, when $x_1>1$.
 The boundary consist of two parts $C=C_1\cup C_2$, where
\begin{equation}
 C_1=\{(t,x); \, x_1<1,\,\,\, (x_1-1)^2+x_2^2/4+x_3^2/4=(1-t)^2\},
\end{equation}
and
\begin{equation}
 C_2=\{(t,x); \, x_1\geq 1,\,\,\, (x_1-1)^2/4+x_2^2/4+x_3^2/4=(1-t)^2\}.
\end{equation}

A conormal to $C_1$ is given by
$$
n=2(1-t)dt -2(1-x_1)dx_1 +x_2 dx_2/2+x_3 dx_3/2,
$$
or expressed in the $L,\bL,A,B$ coordinates $u=(t-x_1)/2$ and
$v=(t+x_1)/2$
$$
n=4(1-v)du-4u dv +x_2 dx_2/2+x_3 dx_3/2.
$$
Hence
\begin{multline*}
g_{\alpha\beta}N^\alpha N^\beta
=g^{\alpha\beta} n_\alpha n_\beta \\=g^{uu} n_u n_u +g^{vv} n_v n_v +2g^{uv} n_u n_v +2g^{uA} n_u n_A+g^{AB} n_A n_B\\
=16(1-v)u+\Big(\frac{1}{4}(u+v)+\frac{1}{64}\Big( x_2^2(1+\chi_1)+x_3^2(1+\chi_1)^{-1}\Big)\tilde{\chi}_2\Big)\tilde{\chi}_2 16 u^2\\
-(x_2^2+x_3^2)(1-v)\tilde{\chi}_2 +(1+\chi_1)^{-1} x_2^2 \frac{1}{4} +(1+\chi_1) x_3^2\frac{1}{4}\\
=16(1-v)u+\frac{1}{4}(x_2^2+x_3^2)+\frac{1}{4}  \big(x_3^2-(1+\chi_1)^{-1} x_2^2\big)\chi_1-(x_2^2+x_3^2)(1-v)\tilde{\chi}_2\\
+\Big(\frac{1}{4}(u+v)+\frac{1}{64}\Big( x_2^2(1+\chi_1)+x_3^2(1+\chi_1)^{-1}\Big)\tilde{\chi}_2\Big)\tilde{\chi}_2 16 u^2.
\end{multline*}
The surface is in the $uv$ coordinates given by
$$
4(1-v)u +x_2^2/4+x_3^2/4=0.
$$
Therefore it is clear that if $N$ is the normal to $C_1$ then
$$
g_{\alpha\beta}N^\alpha N^\beta\leq 0,
$$ with equality only if $u=0$. This proves that $C_1$ is space like apart from when $u=0$ where its null.

\section*{Acknowledgments}
H.L \ is partially supported by NSF grant DMS--1237212
%\G{L}{B}{C}
%\end{align*}
%\bibliography{EinsCount}
%\bibliographystyle{amsplain}
\providecommand{\bysame}{\leavevmode\hbox to3em{\hrulefill}\thinspace}
\providecommand{\MR}{\relax\ifhmode\unskip\space\fi MR }
\providecommand{\MRhref}[2]{%
  \href{http://www.ams.org/mathscinet-getitem?mr=#1}{#2}
}
\providecommand{\href}[2]{#2}

\end{document}